\documentclass[11pt,a4paper]{article}
\usepackage{comment}
\usepackage{calc}
\usepackage{tikz}
\usepackage{epsf,epsfig,amsfonts,amsgen,amsmath,amssymb,amstext,amsbsy,amsopn,amsthm}
\usepackage{color}
\usepackage{graphicx}
\usepackage{caption}
\usepackage{epstopdf}
\usepackage{subfigure}
\usepackage{bm}
\newtheorem{thm}{Theorem}

\newtheorem{lem}{Lemma}

\theoremstyle{definition}

\newtheorem{claim}{Claim}
\newtheorem{case}{Case}
\newtheorem{remark}[claim]{Remark}

\begin{document}

\title
{Rainbow triangles in arc-colored tournaments
\thanks{The first author is supported by NSFC (Nos. 11301098 and 11601428) and GXNSF (No. 2016GXNSFFA380011); the second author is supported by  NSFC (Nos. 11671320 and U1803263);
the third author is supported by NSFC (No. 11601430), the Fundamental Research Funds for the Central Universities (No.
3102019ghjd003) and China Postdoctoral Science Foundation (No. 2016M590969) and the fourth author is supported by NSFC (No. 11901459).}}
\author{\quad Wei Li $^{a, b}$\thanks{Corresponding author. E-mail addresses: muyu.yu@163.com (W. Li), sgzhang@nwpu.edu.cn (S. Zhang), bai@nwpu.edu.cn (Y. Bai), liruonan@mail.nwpu.edu.cn (R. Li). },
\quad Shenggui Zhang $^{b, c}$,
\quad Yandong Bai $^{b}$,
\quad Ruonan Li $^{b}$\\[2mm]
\small $^{a}$ College of Mathematics and Statistics, Guangxi Normal University, \\
\small Guilin, Guangxi  541004, P.R. China\\
\small $^{b}$ School of Mathematics and Statistics, Northwestern Polytechnical University,  \\
\small Xi'an, Shaanxi  710029, P.R. China\\
\small $^{c}$ Xi¡¯an-Budapest Joint Research Center for Combinatorics, Northwestern Polytechnical University, \\
\small Xi'an, Shaanxi  710129, P.R. China\\}
\date{\today}

\maketitle

\begin{abstract}
Let $T_{n}$ be an arc-colored tournament of order $n$. The maximum monochromatic indegree $\Delta^{-mon}(T_{n})$ (resp. outdegree $\Delta^{+mon}(T_{n})$) of $T_{n}$ is the maximum number of in-arcs (resp. out-arcs) of a same color incident to a vertex of $T_{n}$. The irregularity $i(T_{n})$ of $T_{n}$ is the maximum difference between the indegree and outdegree of a vertex of $T_{n}$. A subdigraph $H$ of an arc-colored digraph $D$ is called rainbow if each pair of arcs in $H$ have distinct colors. In this paper, we show that each vertex $v$ in an arc-colored tournament $T_{n}$ with $\Delta^{-mon}(T_n)\leq\Delta^{+mon}(T_n)$ is contained in at least $\frac{\delta(v)(n-\delta(v)-i(T_n))}{2}-[\Delta^{-mon}(T_{n})(n-1)+\Delta^{+mon}(T_{n})d^+(v)]$ rainbow triangles, where $\delta(v)=\min\{d^+(v), d^-(v)\}$. We also give some maximum monochromatic degree conditions for $T_{n}$ to contain rainbow triangles, and to contain rainbow triangles passing through a given vertex. Finally, we present some examples showing that some of the conditions in our results are best possible.


\smallskip

\end{abstract}

\section{Introduction}

In this paper we only consider finite and simple graphs and digraphs, i.e. without loops or multiple edges (arcs). A cycle in a digraph always means a {\it directed cycle}. We use \cite{Bang-Jensen: 2001} and \cite{Bondy: 2008} for terminology and notations not defined here.

Let $G=\left(V(G), E(G)\right)$ be a graph. An {\it edge-coloring} of $G$ is a mapping $C: E(G) \rightarrow \mathbb{N}$, where $\mathbb{N}$ is the set of natural numbers. We call $G$ an {\it edge-colored graph}, if it has an edge-coloring.
We use $C(G)$ to denote the set  of colors appearing on the edges of $G$.
The {\it maximum monochromatic degree} $\Delta^{mon}(G)$ of $G$ is the maximum number of edges of a same color incident to a vertex of $G$. For a vertex $v$ of $G$, the {\it color degree} $d^{c}_{G}(v)$ of $v$ is the number of colors assigned on the edges incident to $v$. The {\it minimum color degree} $\delta^{c}(G)$ is the minimum $d^{c}_{G}(v)$ over all vertices $v$ of $G$. A subgraph $H$ of $G$ is called {\it rainbow} if all edges of $H$ have distinct colors.

The existence of rainbow subgraphs has been widely studied, readers can see the survey papers \cite{Fujita: 2014, X.Li: 2008}. In particular, the existence of rainbow triangles attracts much attention during the past decades. For an edge-colored complete graph $K_n$, Gallai \cite{Gallai: 1967} and Fox et al. \cite{Fox: 2015} characterized the coloring structure of $K_n$ without containing rainbow triangles. Balogh et al. \cite{Balogh: 2017} obtained the
maximum number of rainbow triangles in $3$-edge-colored
complete graphs.  Gy\'{a}rf\'{a}s and Simonyi \cite{Gyarfas: 2004} proved that each edge-colored $K_n$ with $\Delta^{mon}(K_n)<\frac{2n}{5}$ contains a rainbow triangle and this bound is tight. Fujita et al. \cite{FLZ: 2017} proved that each edge-colored $K_n$ with $\delta^c(K_{n})> \log_{2}n$ contains a rainbow triangle and this bound is tight. For a general edge-colored graph $G$ of order $n$, Li and Wang \cite{Li-Wang: 2012} proved that if $\delta^{c}(G)\geq\frac{\sqrt{7}+1}{6}n$, then $G$ has a rainbow triangle. Li \cite{H.Li: 2013} and Li et al. \cite{B.Li: 2014} improved the condition to $\delta^{c}(G)>\frac{n}{2}$ independently, and showed that this bound is tight. Li et al. \cite{Li-Ning-Zhang: 2016} proved that if $G$ is an edge-colored graph of order $n$ satisfying $d^c(u)+d^c(v)\geq n+1$ for every edge $uv\in E(G)$, then it contains a rainbow triangle. Li et al. \cite{B.Li: 2014} proved that if $G$ is an edge-colored graph of order $n$ with $|E(G)|+|C(G)|\geq \frac{n(n+1)}{2}$, then it contains a rainbow triangle. Fujita et al. \cite{FNXZ: 2019} characterized all graphs $G$ satisfying $|E(G)|+|C(G)|\geq \frac{n(n+1)}{2}-1$ but containing no rainbow triangles. Ehard et al. \cite{Ehard-Mohr: 2020} proved that if $G$ is an edge-colored graph of order $n$ with $|E(G)|+|C(G)|\geq \frac{n(n+1)}{2}+k-1$, then it contains at least $k$ rainbow triangles.  Hoppen et al. \cite{HLO: 2017} characterized the graphs with the largest number of
edge-colorings avoiding a rainbow triangle. Aharoni et al. \cite{ADH: 2019} determined the
maximum number of edges in an $n$-vertex edge-colored
graph where all color classes have size at most $k$ and there
is no rainbow triangle. Jin et al. \cite{Jin-Wang-Wang-Lv: 2020} studied rainbow triangles in edge-colored Kneser graphs. For more results on rainbow cycles, we recommend \cite{Albert: 1995, Erdos: 1983, Frieze: 1993, Hahn: 1986} .

Motivated by the fruitful results on the existence of rainbow triangles in undirected graphs, we propose the problem on the existence of rainbow triangles in digraphs. Before proceeding, we first give some terminology and notations in digraphs.

Let $D=\left(V(D), A(D)\right)$ be a digraph. If $uv\in A(D)$, then we say that $u$ {\it dominates } $v$ (or $v$ is {\it dominated} by $u$). For a vertex $v$ of $D$, the {\it in-neighborhood } $N^{-}_{D}(v)$ of $v$ is the set of vertices dominating $v$, and  the {\it out-neighborhood } $N^{+}_{D}(v)$ of $v$ is the set of vertices dominated by $v$. The {\it indegree} $d^{-}_{D}(v)$ and {\it outdegree} $d^{+}_{D}(v)$ of $v$ are the cardinality of $N^{-}_{D}(v)$ and $N^{+}_{D}(v)$, respectively. Let $\delta_{D}(v)=\min\{d^-_{D}(v), d^+_{D}(v)\}$. The {\it maximum indegree $\Delta^{-}(D)$} (resp. {\it maximum outdegree $\Delta^{+}(D)$}) and the {\it minimum indegree $\delta^{-}(D)$} (resp. {\it minimum outdegree $\delta^{+}(D)$}) of $D$, is the maximum $d^{-}_{D}(v)$ (resp. $d^{+}_{D}(v)$) and the minimum $d^{-}_{D}(v)$ (resp. $d^{+}_{D}(v)$) over all vertices $v$ of $D$, respectively. The digraph $D$ is {\it strongly connected} if for every pair of distinct vertices $u$, $v$ in $D$, there exists a $(u, v)$-path. The subdigraph of $D$ induced by $S\subseteq V(D)$ is denoted by $D[S]$. For two disjoint subsets $X$ and $Y$ of $V(D)$, we use $A_{D}(X, Y)$ to denote the set of arcs from $X$ to $Y$. An {\it arc-coloring } of $D$ is a mapping $C$: $A(D)\rightarrow \mathbb{N}$, where $\mathbb{N}$ is the natural number set. We call $D$ an {\it arc-colored} digraph if it has an arc-coloring. We use $C(D)$ and $c(D)$ to denote the set and the number of colors appearing on the arcs of $D$, respectively. For a nonempty subset $S$ of $V(D)$, the {\it maximum monochromatic indegree} (resp. {\it maximum monochromatic outdegree}) of $S$, denoted by $\Delta^{-mon}_{D}(S)$ (resp. $\Delta^{+mon}_{D}(S)$), is the maximum number of in-arcs (resp. out-arcs) of a same color incident to a vertex $v\in S$. We write $\Delta^{-mon}(D)$ for $\Delta^{-mon}_{D}(V(D))$ and $\Delta^{+mon}(D)$ for $\Delta^{+mon}_{D}(V(D))$, respectively.
If there is no ambiguity, we often omit the subscript $D$ in the above notations. When a set contains only one element $s$, we often write $s$ instead of $\{s\}$. A digraph is called {\it rainbow } if all of its arcs have distinct colors. A digraph, which is not rainbow, is called {\it non-rainbow}.

A {\it tournament} is a digraph such that each pair of vertices are joined by precisely one arc. We use $T_{n}$ to denote a tournament of order $n$. The {\it irregularity} $i(T_{n})$ of $T_{n}$ is the maximum difference between the indegree and outdegree of a vertex. A tournament $T_{n}$ is said to be {\it regular} if $i(T_{n})=0$, and {\it almost regular} if $i(T_{n})=1$. Since $d^{+}(v)+d^{-}(v)=n-1$ for each vertex $v\in V(T_{n})$, we have $\max\{\Delta^{-}(T_{n}), \Delta^{+}(T_{n})\}\leq \frac{n+i(T_{n})-1}{2}$ and $\min\{\delta^{-}(T_{n}), \delta^{+}(T_{n})\}\geq\frac{n-i(T_{n})-1}{2}$.

For research on arc-colored tournaments, see \cite{Bai-Fujita-Zhang: 2018, Bai-Li-Zhang: 2018, Bland: 2016}. In this paper we focus on the existence and enumeration of rainbow triangles in strongly connected arc-colored tournaments. By symmetry, throughout this paper we can assume that $\Delta^{-mon}(T_n)\leq\Delta^{+mon}(T_n)$.

We first consider the number of rainbow triangles passing through a given vertex in an arc-colored tournament.

\begin{thm}\label{thm1}
Let $T_{n}$ be a strongly connected arc-colored tournament with $\Delta^{-mon}(T_n)\leq\Delta^{+mon}(T_n)$. Then for each vertex $v$ of $T_n$, the number of rainbow triangles containing $v$ is at least
$$\frac{\delta(v)\left(n-\delta(v)-i(T_n)\right)}{2}-\left[\Delta^{-mon}(T_{n})(n-1)+\Delta^{+mon}(T_{n})d^+(v)\right].$$
\end{thm}

\begin{remark}\label{vertex}
The bound in Theorem \ref{thm1} is tight. Let $T_n$ be an arc-colored tournament of order $n=4k-1+2i(T_n)$, where $k$ is a positive integer. Let $v$ be a vertex of $T_n$ with $d^+(v)=d^-(v)=2k-1+i(T_n)$, $W=N^{+}(v)=\{w_{1}, w_{2}, \ldots, w_{2k-1+i(T_n)}\}$ and $U=N^{-}(v)=\{u_{1}, u_{2}, \ldots, u_{2k-1+i(T_n)}\}$. Let $T_n[W]$ and $T_n[U]$ be two regular tournaments of order $2k-1+i(T_n)$. For each vertex $w_{j}\in W$, let $N^{+}(w_{j})\cap U=\{u_{j}, \ldots, u_{j+k-1}\}$. Then $T_n$ is strongly connected. For each $w_{j}$, $1\leq j\leq 2k-1+i(T_n)$, let $C(vw_{j})=j$, $C(u_{j}v)=j$, $C(w_{j}u_{j+1})=j$ and $C(w_{j}u_{j+2})=j+2$. Finally, color the remaining arcs with distinct new colors. Then
$\Delta^{+mon}(T_n)=\Delta^{-mon}(T_n)=1$, and the number of rainbow triangles containing $v$ is
$$
(k-3)\left(2k-1+i(T_{n})\right)=\frac{\delta(v)\left(n-\delta(v)-i(T_n)\right)}{2}-\left[\Delta^{-mon}(T_{n})(n-1)+\Delta^{+mon}(T_{n})d^+(v)\right].
$$
\end{remark}

Based on Theorem \ref{thm1}, we give some maximum monochromatic degree conditions so that every vertex in $T_{n}$ is contained in a rainbow triangle.

\begin{thm}\label{thm1+}
Let $T_{n}$ be a strongly connected arc-colored tournament with $\Delta^{-mon}(T_n)\leq\Delta^{+mon}(T_n)$ and
\begin{equation*}
  2\Delta^{-mon}(T_{n})+\Delta^{+mon}(T_{n})\leq\left\{
   \begin{aligned}
   &\frac{(n-1-i(T_{n}))(n+1-i(T_{n}))}{4(n-1+i(T_{n}))}, ~~~&1\leq i(T_{n})< \frac{n+3}{3};\\
   &\frac{n-2i(T_{n})}{4}, ~~~&\text{otherwise}.\\
   \end{aligned}
   \right.
  \end{equation*}
Then every vertex of $T_n$ is contained in a rainbow triangle
\end{thm}

\begin{remark}\label{rem1}
When $i(T_{n})=0$ or $\frac{n+3}{3}\leq i(T_{n})\leq n-3$, the bounds in Theorem \ref{thm1+} are tight. When $1\leq i(T_{n})< \frac{n+3}{3}$, the gap between the bound in Theorem \ref{thm1+} and the best possible bound is at most $1$ (See Section \ref{sec:3} for examples).
\end{remark}

We also investigate the existence of rainbow triangles in $T_n$.

\begin{thm}\label{thm2}
Let $T_{n}$ be a strongly connected arc-colored tournament with $\Delta^{-mon}(T_n)\leq\Delta^{+mon}(T_n)$. If $n$ is odd and $\Delta^{-mon}(T_{n})<\frac{n^{2}+n-3i(T_{n})^{2}}{12n}$, or $n$ is even and $\Delta^{-mon}(T_{n})<\frac{n^{2}-1-3i(T_{n})^{2}}{12(n-1)}$, then there exists a rainbow triangle in $T_n$.
\end{thm}

\begin{remark}\label{rem3}
The bounds in Theorem \ref{thm2} may not be tight. However, for regular tournaments the best possible bound can not be larger than $\frac{n}{11}$ (See Section \ref{sec:3} for examples).
\end{remark}

We prove Theorems \ref{thm1}, \ref{thm1+} and \ref{thm2} in Section 2. In Section 3, we give some examples to analyze the tightness of the bounds in Theorems \ref{thm1+} and \ref{thm2}. In Section 4, we propose some further research problems.

\section{The proofs}

Let $v$ be a vertex in an arc-colored tournament $T_n$. We use $n(\overrightarrow{C_{3}}, v)$ to denote the number of triangles containing $v$, and $n(R\overrightarrow{C_{3}}, v)$ and $n(NR\overrightarrow{C_{3}}, v)$ to denote the number of rainbow and non-rainbow triangles containing $v$, respectively. Then $$n(R\overrightarrow{C_{3}}, v)=n(\overrightarrow{C_{3}}, v) - n(NR\overrightarrow{C_{3}}, v).$$
So by estimating the minimum number of triangles and the maximum number of non-rainbow triangles passing through the vertex $v$, we will get a lower bound for $n(\overrightarrow{C_{3}}, v)$.

\begin{lem}\label{lemma1}
Let $T_{n}$ be a strongly connected tournament of order $n$ and irregularity $i(T_n)$. Then for each vertex $v$ of $T_n$, we have
$$n(\overrightarrow{C_{3}}, v)\geq\frac{\delta(v)(n-\delta(v)-i(T_n))}{2}.$$
\end{lem}

\begin{proof}
Let $v$ be a vertex of $T_n$. Let $W=N^{+}(v)$ and $U=N^{-}(v)$. Then $V(T_n)=W\cup U\cup \{v\}$ and
$n(\overrightarrow{C_{3}}, v)=|A(W, U)|$. By the definition of irregularity, we have
 $$-|W|i(T_n)\leq\sum_{w\in W}d^+(w)-\sum_{w\in W}d^-(w)=|A(W,U)|-|A(U,W)|-|W|\leq |W|i(T_n).$$
 Note that
 $$
 |A(W,U)|+|A(U,W)|=|W||U|=|W|(n-1-|W|).
 $$
We have
\begin{equation}\label{eq1}
  |A(W,U)|\geq\frac{|W|(n-|W|-i(T_n))}{2}=\frac{d^+(v)(n-d^+(v)-i(T_n))}{2}.
\end{equation}
Similarly, since
$$
-|U|i(T_n)\leq \sum_{u\in U}d^+(u)-\sum_{u\in U}d^-(u)=|A(U, W)|+|U|-|A(W, U)|\leq |U|i(T_n),
$$
we have
\begin{equation}\label{eq2}
 |A(W,U)|\geq\frac{(n-1-|W|)(|W|+1-i(T_n))}{2}=\frac{d^-(v)(n-d^-(v)-i(T_n))}{2}.
\end{equation}
Combining Inequalities (\ref{eq1}) and (\ref{eq2}), we get
$$
 |A(W,U)|\geq\frac{1}{2}\max\left\{{d^+(v)(n-d^+(v)-i(T_n))}, {d^-(v)(n-d^-(v)-i(T_n))}\right\}.
$$
By easy calculation, we can get
$$
n(\overrightarrow{C_{3}}, v)=|A(W, U)|\geq\frac{\delta(v)(n-\delta(v)-i(T_n))}{2}.
$$
This completes the proof of Lemma \ref{lemma1}.
\end{proof}

Next we will estimate the maximum number of non-rainbow triangles passing through a given vertex $v$. A triangle with vertex set $\{v, w, u\}$ and arc set $\{vw, wu, uv\}$ is denoted by $T_{vwu}$.
\begin{lem}\label{lemma2}
Let $T_{n}$ be a strongly connected arc-colored tournament of order $n$. Then for each vertex $v$ of $T_n$, we have
$$
n(NR\overrightarrow{C_{3}}, v)\leq\Delta^{-mon}(T_{n})(n-1)+\Delta^{+mon}(T_{n})d^+(v).
$$
\end{lem}

\begin{proof}
Let $v$ be a vertex of $T_{n}$ and set $W=N^{+}(v)$ and $U=N^{-}(v)$. Define $W_{i}=\{w\in W|C(vw)=i\}$ and $U_{j}=\{u\in U|C(uv)=j\}$. Then $\sum_{i\in C(T_n)}|W_{i}|=d^{+}(v)$ and $\sum_{j\in C(T_n)}|U_{j}|=d^{-}(v)$. For a triangle $T_{vwu}$, if it is not rainbow, then it must belong to at least one of the sets $S_{1}(v)=\{T_{vwu}|C(vw)=C(wu)\}$, $S_{2}(v)=\{T_{vwu}|C(wu)=C(uv)\}$ and $S_{3}(v)=\{T_{vwu}|C(uv)=C(vw)\}$. Denote the cardinality of $S_{i}(v)$ by $t_{i}(v)$, $i=1, 2, 3$. Then
$$n(NR\overrightarrow{C_{3}}, v)\leq t_{1}(v)+t_{2}(v)+t_{3}(v).$$
Let $X_i=\{wu\in A(W_{i},U): C(wu)=i\}$. Then we have
$$t_{1}(v)=\sum_{i\in C(T_n)}|X_i|\leq\sum_{i\in C(T_n)} |W_{i}|\Delta^{+mon}(T_{n})=\Delta^{+mon}(T_{n})d^+(v).$$
Let $Y_j=\{wu\in A(W,U_j): C(wu)=j\}$. Then we have
$$
t_{2}(v)=\sum_{j\in C(T_n)}|Y_j|\leq\sum_{j\in C(T_n)} |U_{j}|\Delta^{-mon}(T_{n})=\Delta^{-mon}(T_{n})d^-(v).
$$
It is not hard to see that
$$t_{3}(v)\leq\sum_{k\in C(T_{n})}|W_{k}||U_{k}|.$$
Note that 
$|U_{k}|\leq\Delta^{-mon}(T_{n})$ for all $k\in C(T_{n})$. We have
$$
\sum_{k\in C(T_{n})}|W_{k}||U_{k}|\leq\Delta^{-mon}(T_{n})\sum_{k\in C(T_{n})}|W_{k}|=\Delta^{-mon}(T_{n})d^{+}(v).
$$
Then
\begin{equation*}
\begin{aligned}
n(NR\overrightarrow{C_{3}}, v)\leq& t_{1}(v)+t_{2}(v)+t_{3}(v)\\
\leq&\Delta^{+mon}(T_{n})d^{+}(v)+\Delta^{-mon}(T_{n})d^{-}(v)+\Delta^{-mon}(T_{n})d^{+}(v)\\
=&\Delta^{-mon}(T_{n})(n-1)+\Delta^{+mon}(T_{n})d^+(v).
\end{aligned}
\end{equation*}
This completes the proof of Lemma \ref{lemma2}.
\end{proof}


\begin{proof}[\bf Proof of Theorem \ref{thm1}]
Since $n(R\overrightarrow{C_{3}}, v)=n(\overrightarrow{C_{3}}, v)-n(NR\overrightarrow{C_{3}}, v)$, Theorem \ref{thm1} follows immediately from  Lemmas \ref{lemma1} and \ref{lemma2}.
\end{proof}

\setcounter{claim}{0}
Note that if $n(R\overrightarrow{C_{3}},v)>0$, then $v$ is contained in a rainbow triangle.

\noindent\textbf{Proof of Theorem \ref{thm1+}.}
Let $v$ be a vertex of $T_n$. The lower bound of $n(\overrightarrow{C_{3}}, v)$ in Lemma \ref{lemma1} is related to $\delta(v)$. Now, we will give a new lower bound of $n(\overrightarrow{C_{3}},v)$ without using $\delta(v)$.
\begin{claim}\label{claim1}
\begin{equation*}
n(\overrightarrow{C_{3}}, v)\geq\left\{
\begin{aligned}
&\frac{n(n-2)}{8}, &\text{if~}i(T_{n})=1; \\
&\frac{(n-1)(n+1-2i(T_{n}))}{8}, &\text{otherwise}.\\
\end{aligned}
\right.
\end{equation*}
\end{claim}
\begin{proof}
Let $f(\delta(v))=\frac{\delta(v)(n-\delta(v)-i(T_n))}{2}$. Since $f'(\delta(v))=\frac{n-2\delta(v)-i(T_n)}{2}$, we can see that $f(\delta(v))$ increases when $\delta(v)\leq\frac{n-i(T_n)}{2}$ and decreases when $\delta(v)\geq\frac{n-i(T_n)}{2}$. Note that $\frac{n-1-i(T_n)}{2}\leq \delta(v)\leq \frac{n-1}{2}$. So the minimum value of $f(\delta(v))$ can only be obtained when $\delta(v)=\frac{n-1-i(T_n)}{2}$ or $\frac{n-1}{2}$.
Comparing $f(\frac{n-1-i(T_n)}{2})$ and $f(\frac{n-1}{2})$, we can prove Claim \ref{claim1}.
\end{proof}

We divide the rest of the proof into three cases according to the irregularity of $T_n$.

\begin{case}
	$i(T_n)=0$.
\end{case}
In this case, $T_n$ is a regular tournament. Since  $2\Delta^{-mon}(T_{n})+\Delta^{+mon}(T_{n})\leq\frac{n}{4}$, by Lemma \ref{lemma2} and Claim \ref{claim1}, we have
\begin{equation*}
\begin{aligned}
n(NR\overrightarrow{C_{3}}, v)\leq\Delta^{-mon}(T_{n})(n-1)+\Delta^{+mon}(T_{n})d^+(v)\leq\frac{n(n-1)}{8}<\frac{n^{2}-1}{8}\leq n(\overrightarrow{C_{3}}, v).
\end{aligned}
\end{equation*}
Thus, $n(R\overrightarrow{C_{3}},v)>0$, namely, $v$ is contained in a rainbow triangle.
\begin{case}
	$\frac{n+3}{3}\leq i(T_n)\leq n-3$.
\end{case}
If  $d^+(v)<\frac{n-1}{2}$, then by Lemma \ref{lemma2} we have
\begin{equation*}
\begin{aligned}
n(NR\overrightarrow{C_{3}}, v)&\leq\Delta^{-mon}(T_{n})(n-1)+\Delta^{+mon}(T_{n})d^+(v)\\
&\leq\frac{(n-1)(n-2i(T_n))}{8}\\
&<\frac{(n-1)(n+1-2i(T_n))}{8}\leq n(\overrightarrow{C_{3}}, v).
\end{aligned}
\end{equation*}
So $v$ is contained in a rainbow triangle.

If $d^+(v)\geq\frac{n-1}{2}$, then $d^+(v)=n-1-\delta(v)$. By Lemma \ref{lemma2} we have
\begin{equation*}
\begin{aligned}
n(NR\overrightarrow{C_{3}}, v)&\leq\Delta^{-mon}(T_{n})(n-1)+\Delta^{+mon}(T_{n})d^+(v)\\
&\leq d^+(v)(2\Delta^{-mon}(T_{n})+\Delta^{+mon}(T_{n})).
\end{aligned}
\end{equation*}
Now we will show $$d^+(v)(2\Delta^{-mon}(T_{n})+\Delta^{+mon}(T_{n}))<\frac{\delta(v)(n-\delta(v)-i(T_{n}))}{2}.$$ It suffices to prove that $$\frac{\delta(v)(n-\delta(v)-i(T_{n}))}{n-1-\delta(v)}> 4\Delta^{-mon}(T_{n})+2\Delta^{+mon}(T_{n}).$$
Define $$g(\delta(v))=\frac{\delta(v)(n-\delta(v)-i(T_{n}))}{n-1-\delta(v)}.$$
 Since
\begin{equation*}
\begin{aligned}
g''(\delta(v))=-\frac{2(n-1)(i(T_{n})-1)}{(n-1-\delta(v))^3}<0,
\end{aligned}
\end{equation*}
we can see that $g(\delta(v))$ is convex when $\frac{n-1-i(T_{n})}{2}\leq \delta(v)\leq\frac{n-1}{2}$. So the minimum value of $g(\delta(v))$ can only be obtained when $\delta(v)=\frac{n-1-i(T_n)}{2}$ or $\frac{n-1}{2}$. Comparing $g(\frac{n-1-i(T_{n})}{2})$ and $g(\frac{n-1}{2})$, we have
$$
\min_{\frac{n-1-i(T_{n})}{2}\leq \delta(v)\leq\frac{n-1}{2}}g(\delta(v))=g(\frac{n-1}{2})=\frac{n+1-2i(T_{n})}{2},
$$
and thus
$$
2\Delta^{-mon}(T_{n})+\Delta^{+mon}(T_{n})\leq\frac{n-2i(T_{n})}{4}<\frac{n+1-2i(T_{n})}{4}\leq\frac{\delta(v)(n-\delta(v)-i(T_{n}))}{2(n-1-\delta(v))}.
$$
Therefore,
$$
n(NR\overrightarrow{C_{3}},v)\leq d^+(v)(2\Delta^{-mon}(T_{n})+\Delta^{+mon}(T_{n}))<\frac{\delta(v)(n-\delta(v)-i(T_{n}))}{2}\leq n(\overrightarrow{C_{3}},v).
$$
This implies that $v$ is contained in a rainbow triangle.

\begin{case}
	$1\leq i(T_{n})<\frac{n+3}{3}$.
\end{case}

The proof of Case 3 is similar to that of Case 2. Note that in this case we have
$$\frac{(n-1-i(T_{n}))(n+1-i(T_{n}))}{2(n-1+i(T_{n}))}<\frac{n+1-2i(T_{n})}{2}.$$

If $d^+(v)\leq\frac{n-1}{2}$, then we have
\begin{equation*}
\begin{aligned}
n(NR\overrightarrow{C_{3}}, v)&\leq\frac{n-1}{2}(2\Delta^{-mon}(T_{n})+\Delta^{+mon}(T_{n}))\\
&\leq\frac{n-1}{2}\frac{(n-1-i(T_{n}))(n+1-i(T_{n}))}{4(n-1+i(T_{n}))}\\
&<\frac{(n-1)(n+1-2i(T_n))}{8}\leq n(\overrightarrow{C_{3}}, v) \text{~(for~$1< i(T_n)< \frac{n+3}{3}$)}
\end{aligned}
\end{equation*}
and
\begin{equation*}
\begin{aligned}
n(NR\overrightarrow{C_{3}}, v)&\leq\frac{n-1}{2}(2\Delta^{-mon}(T_{n})+\Delta^{+mon}(T_{n}))\\
&\leq\frac{n-1}{2}\frac{(n-1-i(T_{n}))(n+1-i(T_{n}))}{4(n-1+i(T_{n}))}\\
&=\frac{(n-1)(n-2)}{8}<\frac{n(n-2)}{8}\leq n(\overrightarrow{C_{3}}, v) \text{~(for~$i(T_n)=1$)}.
\end{aligned}
\end{equation*}
This implies that $v$ is contained in a rainbow triangle.

If $d^+(v)>\frac{n-1}{2}$, then we have
$$n(NR\overrightarrow{C_{3}}, v)<d^+(v)(2\Delta^{-mon}(T_{n})+\Delta^{+mon}(T_{n})).$$
Since $i(T_{n})<\frac{n+3}{3}$, we have
$$
\min_{\frac{n-1-i(T_{n})}{2}\leq \delta(v)\leq\frac{n-1}{2}}g(\delta(v))=g(\frac{n-1-i(T_{n})}{2})=\frac{(n-1-i(T_{n}))(n+1-i(T_{n}))}{2(n-1+i(T_{n}))}
$$
and
$$
2\Delta^{-mon}(T_{n})+\Delta^{+mon}(T_{n})\leq\frac{(n-1-i(T_{n}))(n+1-i(T_{n}))}{4(n-1+i(T_{n}))}\leq\frac{\delta(v)(n-\delta(v)-i(T_{n}))}{2(n-1-\delta(v))}.
$$
Thus,
$$
n(NR\overrightarrow{C_{3}},v)<d^+(v)(2\Delta^{-mon}(T_{n})+\Delta^{+mon}(T_{n}))\leq\frac{\delta(v)(n-\delta(v)-i(T_{n}))}{2}\leq n(\overrightarrow{C_{3}},v).
$$
This implies that $v$ is contained in a rainbow triangle, completing the proof of Theorem \ref{thm1+}.    \qed

To prove Theorem \ref{thm2}, we will estimate the minimum number of triangles and the maximum number of non-rainbow triangles in $T_n$, respectively. We use $n(\overrightarrow{C_{3}})$, $n(R\overrightarrow{C_{3}})$ and $n(NR\overrightarrow{C_{3}})$ to denote the number of triangles, rainbow triangles and non-rainbow triangles in $T_n$, respectively. Here we give Lemmas \ref{lemma3} and \ref{lemma4}.

\begin{lem}\label{lemma3}
Let $T_{n}$ be a strongly connected tournament of order $n$ and irregularity $i(T_{n})$. If $n$ is odd, then
$$n(\overrightarrow{C_{3}})\geq\frac{(n-1)(n^{2}+n-3i(T_{n})^{2})}{24}.$$ If $n$ is even, then $$n(\overrightarrow{C_{3}})\geq\frac{n(n^{2}-1-3i(T_{n})^{2})}{24}.$$
\end{lem}
\begin{proof}
Let $v$ be a vertex of $T_n$ and set $W=N^{+}(v)$ and $U=N^{-}(v)$. Then $n(\overrightarrow{C_{3}}, v)=|A(W, U)|$. We can partition the out-neighborhood of a vertex $w \in W$ into two parts. Let $P_{1}(w)=N^{+}(w)\cap W$ and $P_{2}(w)=N^{+}(w)\cap U$. Then $d^{+}(w)=|P_{1}(w)|+|P_{2}(w)|$. Thus,
\begin{equation*}
\begin{aligned}
|A(W, U)|=\sum_{w \in W}|P_{2}(w)|&=\sum_{w \in W}d^{+}(w)-\sum_{w \in W}|P_{1}(w)|\\
&=\sum_{w \in N^+(v)}d^{+}(w)-\frac{d^+(v)(d^+(v)-1)}{2}.
\end{aligned}
\end{equation*}
 Since a triangle contains three vertices, we have
$$
n(\overrightarrow{C_{3}})=\frac{1}{3}\sum_{v\in V(T_n)}n(\overrightarrow{C_{3}}, v)=\frac{1}{3}\sum_{v\in V(T_n)}\left(\sum_{w \in N^+(v)}d^{+}(w)-\frac{d^+(v)(d^+(v)-1)}{2}\right).
$$
For each vertex $w$, we can see that $w$ is an out-neighbor of exactly $d^-(w)$ vertices. So
\begin{equation*}
\begin{aligned}
&\sum_{v\in V(T_n)}\left(\sum_{w \in N^+(v)}d^{+}(w)-\frac{d^+(v)(d^+(v)-1)}{2}\right)\\
=&\sum_{v\in V(T_n)}\left(d^+(v)d^-(v)-\frac{(d^+(v))^2}{2}+\frac{d^+(v)}{2}\right)\\
=&\frac{n(n-1)(2n-1)}{4}-\frac{3}{2}\sum_{v\in V(T_n)}(d^{+}(v))^{2}.
\end{aligned}
\end{equation*}
It suffices to calculate
 $$ \max\sum_{i=1}^{n} (d^{+}(v_{i}))^{2}\qquad\qquad\qquad\qquad$$
 \begin{equation*}
 {\bf s.t.}\left\{
   \begin{aligned}
  &\frac{n-1-i(T_{n})}{2}\leq d^{+}(v_{i})\leq \frac{n-1+i(T_{n})}{2};\\
      &\sum_{i=1}^{n} d^{+}(v_{i})=\frac{n(n-1)}{2}.
   \end{aligned}
   \right.
  \end{equation*}
We claim that if $\sum_{i=1}^{n} (d^{+}(v_{i}))^{2}$ attains the maximum value then the number of vertices with outdegree $\frac{n-1+i(T_{n})}{2}$ is maximum, and subject to this, the number of vertices with outdegree $\frac{n-1-i(T_{n})}{2}$ is maximum. In other words, there exist no two vertices $x$, $y$ with $\frac{n-1-i(T_{n})}{2}<d^{+}(x)\leq d^{+}(y)<\frac{n-1+i(T_{n})}{2}$. Otherwise, we can get a larger $\sum_{i=1}^{n} (d^{+}(v_{i}))^{2}$ by changing $d^{+}(x)$, $d^{+}(y)$ to $d^{+}(x)-1$, $d^{+}(y)+1$.
If $n$ is odd, then
\begin{equation*}
\begin{aligned}
\max\sum_{v\in V(T_n)}(d^{+}(v))^{2}=&\frac{n-1}{2}\left(\frac{n-1+i(T_{n})}{2}\right)^{2}+\left(\frac{n-1}{2}\right)^{2}+\frac{n-1}{2}\left(\frac{n-1-i(T_{n})}{2}\right)^{2}\\
=&\frac{(n-1)\left[(n-1)^{2}+i(T_{n})^{2}+n-1\right]}{4}.
\end{aligned}
\end{equation*}
So we have
\begin{equation*}
\begin{aligned}
n(\overrightarrow{C_{3}})\geq\frac{(n-1)(n^{2}+n-3i(T_{n})^{2})}{24}.
\end{aligned}
\end{equation*}
If $n$ is even, then
\begin{equation*}
\begin{aligned}
\max\sum_{v\in V(T_n)}(d^{+}(v))^{2}=&\frac{n}{2}\left(\frac{n-1+i(T_{n})}{2}\right)^{2}+\frac{n}{2}\left(\frac{n-1-i(T_{n})}{2}\right)^{2}\\
=&\frac{n\left[(n-1)^{2}+i(T_{n})^{2}\right]}{4}.
\end{aligned}
\end{equation*}
So we have
\begin{equation*}
\begin{aligned}
n(\overrightarrow{C_{3}})\geq\frac{n(n^{2}-1-3i(T_{n})^{2})}{24},
\end{aligned}
\end{equation*}
completing the proof.
\end{proof}

\begin{lem}\label{lemma4}
Let $T_{n}$ be a strongly connected arc-colored tournament of order $n$. Then
$$n(NR\overrightarrow{C_{3}})\leq\frac{n(n-1)\Delta^{-mon}(T_n)}{2}.$$
\end{lem}
\begin{proof}
Let $P$ be the set of all monochromatic directed path of length $2$ in $T_n$. For each non-rainbow triangle, there must be a monochromatic directed path of length $2$ in it. For each two distinct non-rainbow triangles, the corresponding monochromatic directed paths of length $2$ are also distinct. So we have $n(NR\overrightarrow{C_{3}})\leq |P|$. Let $v$ be a vertex of $T_{n}$. Let $W=N^{+}(v)$ and $U=N^{-}(v)$. Define $W_{i}=\{w\in W|C(vw)=i\}$ and $U_{j}=\{u\in U|C(uv)=j\}$. Then $\sum_{i\in C(T_n)}|W_{i}|=d^{+}(v)$ and $\sum_{j\in C(T_n)}|U_{j}|=d^{-}(v)$. The number of monochromatic directed paths of length $2$ which contain $v$ as the center is
$$\sum_{k\in C(T_{n})}|W_{k}||U_{k}|.$$
By the proof of Lemma \ref{lemma2}, we know that
$$\sum_{k\in C(T_{n})}|W_{k}||U_{k}|\leq\Delta^{-mon}(T_{n})d^{+}(v).$$
Thus,
$$
n(NR\overrightarrow{C_{3}})\leq|P|\leq\sum_{v\in V(T_n)}\Delta^{-mon}(T_{n})d^{+}(v)=\frac{n(n-1)\Delta^{-mon}(T_n)}{2}.
$$
The proof is complete.
\end{proof}

\noindent\textbf{Proof of Theorem \ref{thm2}.} If $n(NR\overrightarrow{C_{3}})<n(\overrightarrow{C_{3}})$, then there must be a rainbow triangle in $T_n$. So Theorem \ref{thm2} follows from Lemmas \ref{lemma3} and \ref{lemma4} immediately. \qed

\section{Tightness analysis of the bounds in Theorems \ref{thm1+} and \ref{thm2}}\label{sec:3}
In this section, we will give some examples to analyze the tightness of the bounds in Theorems \ref{thm1+} and \ref{thm2}. Examples 1, 2 and 3 are for Theorem \ref{thm1+}. Examples 4 and 5 are for Theorem \ref{thm2}.

Before we give the examples, we will prove an easy but useful result first.
\begin{thm}\label{thm3}
If $T_{n}$ is a regular tournament, then $T_{n}$ must be strongly connected.
\end{thm}
\begin{proof}
Let $T_{n}$ be a regular tournament. Then for every vertex $v$ of $T_{n}$, we have $d^{+}(v)=d^{-}(v)$. By contradiction, assume that $T_{n}$ is not strongly connected. Then the vertex set of $T_{n}$ can be partitioned into two
nonempty subsets $V_{1}$ and $V_{2}$, such that all arcs between $V_{1}$ and $V_{2}$ have the same direction. Without loss of generality, we can assume that $|A(V_{1},V_{2})|=|V_{1}||V_{2}|$ and $|A(V_{2},V_{1})|=0$. Since $T_{n}$ is regular, we have $$\sum_{v\in V_{1}}d^{+}(v)=\sum_{v\in V_{1}}d^{-}(v).$$
Note that $$\sum_{v\in V_{1}}d^{+}(v)=\sum_{v\in V_{1}}d^{+}_{T_{n}[V_{1}]}(v)+|A(V_{1},V_{2})|,$$
$$\sum_{v\in V_{1}}d^{-}(v)=\sum_{v\in V_{1}}d^{-}_{T_{n}[V_{1}]}(v)+|A(V_{2},V_{1})|$$
and $$\sum_{v\in V_{1}}d^{+}_{T_{n}[V_{1}]}(v)=\sum_{v\in V_{1}}d^{-}_{T_{n}[V_{1}]}(v).$$
So we have $|A(V_{1},V_{2})|=|A(V_{2},V_{1})|$, a contradiction.
\end{proof}

\noindent\textbf{Example 1.}\quad This example shows that the upper bound $\frac{n-2i(T_n)}{4}$ in Theorem \ref{thm1+} is tight, for $i(T_n)=0$ or $\frac{n+3}{3}\leq i(T_n)\leq \frac{n-3}{2}$.

We construct a tournament $T_n$ with $n=4m-1$ vertices. Let $v$ be a vertex of $T_n$ with $d^+(v)=d^-(v)=2m-1$, $W=N^{+}(v)=\{w_{1}, w_{2}, \ldots, w_{2m-1}\}$ and $U=N^{-}(v)=\{u_{1}, u_{2}, \ldots, u_{2m-1}\}$. Let $T_n[W]$ and $T_n[U]$ be two regular tournaments. For each vertex $w_{j}\in W$, let $N^{+}(w_{j})\cap U=\{u_{j}, \ldots, u_{j+k-1}\}$ (indices are taken modulo $2m-1$), where $1\leq k \leq m=\frac{n+1}{4}$. By Theorem \ref{thm3}, we can see that $T_n$ is strongly connected.

Now let us see the irregularity of $T_n$. For vertex $v$, we have $|d^+(v)-d^-(v)|=0$. For every vertex $w \in W$, we have $d^{+}(w)=m-1+k$ and $d^{-}(w)=m-1+1+2m-1-k=m-1+2m-k$. Since $k\leq m$, we have
$|d^+(w)-d^-(w)|=2m-2k$. Similarly, for every vertex $u \in U$, we have $d^{+}(u)=m-1+2m-k$, $d^{-}(u)=m-1+k$ and
$|d^+(u)-d^-(u)|=2m-2k$. So, $i(T_n)=2m-2k$. Since $1\leq k \leq m=\frac{n+1}{4}$, we have $0\leq i(T_n)\leq \frac{n-3}{2}$.

Next, we will assign colors to the arcs of $T_n$. For $1\leq j\leq 2m-1$, let $C(vw_{j})=j$, $C(u_{j}v)=j$, $C(w_{j}u_{j+1})=j+1$ and $C(w_{j}u_{p})=j$ for $j+2\leq p\leq j+k-1$. Finally, color the remaining arcs with distinct new colors. Then
$\Delta^{+mon}(T_n)=k-2$ and $\Delta^{-mon}(T_n)=1$. Thus, $$2\Delta^{-mon}(T_n)+\Delta^{+mon}(T_n)=k=\frac{2m-i(T_n)}{2}=\frac{n+1-2i(T_n)}{4},$$ but there is no rainbow triangle containing $v$.

\noindent\textbf{Example 2.} \quad This example shows that if $i(T_{n})\geq \frac{n-1}{2}$, then even the condition $\Delta^{+mon}(T_n)\\=\Delta^{-mon}(T_n)=1$ can not guarantee every vertex in $T_n$ is contained in a rainbow triangle.

We construct a tournament $T_n$ with $n$ vertices. Let $v$ be a vertex of $T_n$ with $d^+(v)=x$ and $d^-(v)=n-1-x$, where $\frac{n+1}{2}\leq x\leq n-2$, $x$ and $n-1-x$ are odd integers. Let $W=N^{+}(v)=\{w_{1}, w_{2}, \ldots, w_{x}\}$ and $U=N^{-}(v)=\{u_{1}, u_{2}, \ldots, u_{n-1-x}\}$. Let $T_n[W]$ and $T_n[U]$ be two regular tournaments. For a vertex $w_{j}\in W$, $1\leq j \leq n-1-x$, let $N^{+}(w_{j})\cap U=\{u_{j}\}$. For a vertex $w_{i}\in W$, $n-x\leq i \leq x$, let $N^{+}(w_{j})\cap U=\emptyset$. By Theorem \ref{thm3}, we can see that $T_n$ is strongly connected.

Now let us see the irregularity of $T_n$. For vertex $v$, since $\frac{n+1}{2}\leq x$, we have $|d^+(v)-d^-(v)|=2x-n+1$. For  vertex $w_{j}\in W$, $1\leq j \leq n-1-x$, we have $d^{+}(w_{j})=\frac{x-1}{2}+1$ and $d^{-}(w_{j})=\frac{x-1}{2}+1+n-1-x-1=\frac{x-1}{2}+n-x-1$. Since $x\leq n-2$, we have
$|d^+(w_{j})-d^-(w_{j})|=n-2-x$. Similarly, for  vertex $w_{i}\in W$, $n-x\leq i \leq x$, we have $d^{+}(w_{i})=\frac{x-1}{2}$, $d^{-}(w_{i})=\frac{x-1}{2}+n-x$ and
$|d^+(w_{i})-d^-(w_{i})|=n-x$. For every vertex $u \in U$, we have $d^{+}(u)= x+\frac{n-x-2}{2}$, $d^{-}(u)= 1+\frac{n-x-2}{2}$ and
$|d^+(u)-d^-(u)|=x-1$. So, $i(T_n)=\max\{2x-n+1, n-2-x, n-x, x-1\}$. Since $\frac{n+1}{2}\leq x\leq n-2$, we have $2x-n+1\leq x-1$ and $n-x\leq \frac{n-1}{2}\leq x-1$. So, $i(T_{n})=x-1\geq \frac{n-1}{2}$.

Next, we will assign colors to the arcs of $T_n$. For $1\leq j\leq x$ and $1\leq i\leq n-1-x$, let $C(vw_{j})=j$ and $C(u_{i}v)=i$. Finally, color the remaining arcs with distinct new colors. Then
$\Delta^{+mon}(T_n)=\Delta^{-mon}(T_n)=1$, but there is no rainbow triangle containing $v$.

\noindent\textbf{Example 3.} For $1\leq i(T_n)<\frac{n+3}{3}$, denote $\lfloor\frac{(n-1-i(T_n))(n+1-i(T_n))}{4(n-1+i(T_n))}\rfloor=m$. Let $T_{n}$ be a tournament of order $n=4k-1+i(T_n)$, where $k$ is a positive integer, $v$ be a vertex of $T_{n}$ with $d^+(v)=2k-1+i(T_n)$ and $d^-(v)=2k-1$, $W=N^{+}(v)=\{w_{1}, w_{2}, \ldots, w_{2k-1+i(T_n)}\}$ and $U=N^{-}(v)=\{u_{1}, u_{2}, \ldots, u_{2k-1}\}$. If $i(T_n)$ is even, then let $T_{n}[W]$ be a regular tournament, otherwise let $T_{n}[W]$ be an almost regular tournament. Let $T_{n}[U]$ be a regular tournament. Since $i(T_n)<\frac{n+3}{3}$, we have $3i(T_n)<n+3=4k+2+i(T_n)$, namely, $i(T_n)\leq 2k$.

\noindent\textbf{Case 1.} $i(T_n)\leq 2k-1$ and $(2k-1)(m+1)+i(T_n)m\geq k(2k-1)$.

We can construct a tournament with $2\Delta^{-mon}(T_{n})+\Delta^{+mon}(T_{n})=m+1$ to show the bound is tight. For $w_{j}$, $1\leq j\leq 2k-1$, let $N^{+}(w_{j})\cap U=\{u_{j}, \ldots, u_{j+m}\}$, and for $w_{2k-1+j}$, $1\leq j\leq i(T_n)$, let $N^{+}(w_{2k-1+j})\cap U=\{u_{j}, u_{j+m+1}, \ldots, u_{j+2m-1}\}$. Since $k\geq 1$ and $i(T_n)\geq 1$, we have
\begin{equation*}
\begin{aligned}
2m-1-(2k-1)=2m-2k&\leq\frac{(n-1-i(T_n))(n+1-i(T_n))}{2(n-1+i(T_n))}-2k\\
&=\frac{2k(4k-2)-2k(4k-2+2i(T_n))}{4k-2+2i(T_n)}\\
&=\frac{-4ki(T_n)}{4k-2+2i(T_n)}\\
&< 0.
\end{aligned}
\end{equation*}
Namely, $2m-1<2k-1$. So
$$
(N^{+}(w_{j})\cap U)\cap(N^{+}(w_{2k-1+j})\cap U)=\{u_{j}\},
$$
and elements in $\{u_{j}, u_{j+m+1}, \ldots, u_{j+2m-1}\}$ are pairwise distinct. We can see that $T_{n}$ is strongly connected.
Let $C(vw_{j})=j$, $C(u_{j}v)=j$, $C(w_{j}u_{j+1})=j+1$ and $C(w_{j}u_{p})=j$ for $1\leq j\leq 2k-1$ and $j+2\leq p\leq j+m$. Let $C(vw_{2k-1+j})=j$, and $C(w_{2k-1+j}u_{q})=j$ for $1\leq j\leq i(T_n)$ and $j+m+1\leq q\leq j+2m-1$. Finally, color the remaining arcs with distinct new colors (See Figure \ref{e3a}). Then
$\Delta^{+mon}(T_{n})=m-1$ and $\Delta^{-mon}(T_{n})=1$. Thus, $$2\Delta^{-mon}(T_{n})+\Delta^{+mon}(T_{n})=m+1=\lceil\frac{(n-1-i(T_n))(n+1-i(T_n))}{4(n-1+i(T_n))}\rceil,$$ but there is no rainbow triangle containing $v$ (indices are taken modulo $2k-1$). So the bound is tight in this case.

\begin{figure}
  \centering
  \includegraphics[width=0.7\textwidth]{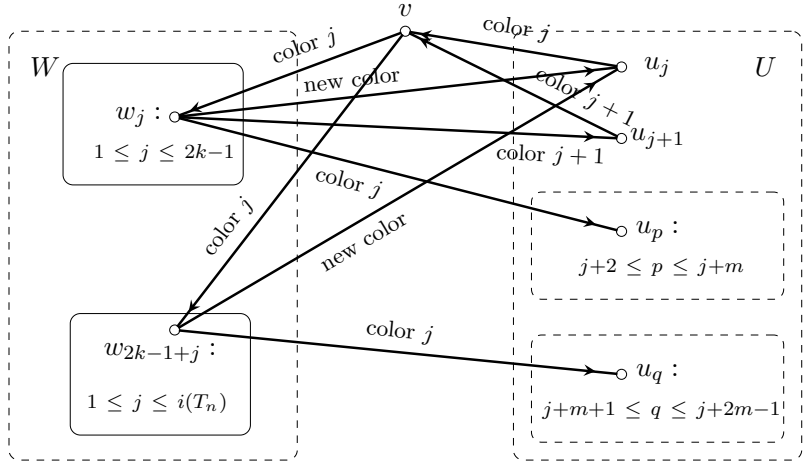}\\
  \caption{A digraph described in case 1 of Example 3.}\label{e3a}
\end{figure}

Note that if $i(T_n)=1$, then $\frac{(n-1-i(T_n))(n+1-i(T_n))}{4(n-1+i(T_n))}=\frac{n-2}{4}=\frac{4k-2}{4}$. Thus, $m=\lfloor\frac{4k-2}{4}\rfloor=k-1$ and
$(2k-1)(m+1)+i(T_n)m=k(2k-1)+k-1> k(2k-1)$.

\noindent\textbf{Case 2.} $i(T_n)= 2k$ or $(2k-1)(m+1)+i(T_n)m< k(2k-1)$.

By the above argument, in this case we have $i(T_n)\geq 2$.
We can construct a tournament with $2\Delta^{-mon}(T_{n})+\Delta^{+mon}(T_{n})=m+2$ and a vertex not contained in any rainbow triangles. Since $$\frac{(n-1-i(T_n))(n+1-i(T_n))}{4(n-1+i(T_n))}<m+1$$ and
\begin{equation*}
\begin{aligned}
(2k-1+i(T_n))\frac{(n-1-i(T_n))(n+1-i(T_n))}{4(n-1+i(T_n))}&=\frac{k(4k-2)(2k-1+i(T_n))(4k-2)}{4k-2+2i(T_n)}\\
&=k(2k-1),
\end{aligned}
\end{equation*}
we have $k(2k-1)<(2k-1+i(T_n))(m+1)$. Namely, $(2k-1+i(T_n))(m+1)-1\geq k(2k-1)$.
For $w_{j}$, $1\leq j\leq 2k-1$, let $N^{+}(w_{j})\cap U=\{u_{j}, \ldots, u_{j+m}\}$, and for $w_{2k-1+j}$, $1\leq j\leq \min\{i(T_n), 2k-1\}$, let $N^{+}(w_{2k-1+j})\cap U=\{u_{j}, u_{j+m+1}, \ldots, u_{j+2m}\}$. If $i(T_n)=2k$, let $N^{+}(w_{4k-1})\cap U=\{u_{1}, \ldots, u_{m}\}$. Since $k\geq 1$ and $i(T_n)\geq 2$, we have
\begin{equation*}
\begin{aligned}
2m-(2k-1)=2m-2k+1&\leq\frac{(n-1-i(T_n))(n+1-i(T_n))}{2(n-1+i(T_n))}-2k+1\\
&=\frac{2k(4k-2)-2k(4k-2+2i(T_n))+4k-2+2i(T_n)}{4k-2+2i(T_n)}\\
&=\frac{4k-2+2i(T_n)-4ki(T_n)}{4k-2+2i(T_n)}\\
&=\frac{(4k-2)(1-i(T_n))}{4k-2+2i(T_n)}\\
&< 0.
\end{aligned}
\end{equation*}
Namely, $2m<2k-1$. So
$$
(N^{+}(w_{j})\cap U)\cap(N^{+}(w_{2k-1+j})\cap U)=\{u_{j}\}.
$$
and elements in $\{u_{j}, u_{j+m+1}, \ldots, u_{j+2m}\}$ are pairwise distinct. We can see that $T_{n}$ is strongly connected.
Let $C(vw_{j})=j$, $C(u_{j}v)=j$, $C(w_{j}u_{j+1})=j+1$ and $C(w_{j}u_{p})=j$, for $1\leq j\leq 2k-1$ and $j+2\leq p\leq j+m$. Let $C(vw_{2k-1+j})=j$, and $C(w_{2k-1+j}u_{q})=j$, for $1\leq j\leq \min\{i(T_n), 2k-1\}$ and $j+m+1\leq q\leq j+2m$. If $i(T_n)=2k$, let $C(vw_{4k-1})=2k$ and $C(w_{4k-1}u_{s})=2k$, for $1\leq s\leq m$. Finally, color the remaining arcs with distinct new colors (See Figure \ref{Figure3}). Then
$\Delta^{+mon}(T_{n})=m$ and $\Delta^{-mon}(T_{n})=1$. Thus, $$2\Delta^{-mon}(T_{n})+\Delta^{+mon}(T_{n})=m+2,$$ but there is no rainbow triangle containing $v$ (indices are taken modulo $2k-1$).

\begin{figure}[h]
	\centering
		{\includegraphics[width=0.7\textwidth]{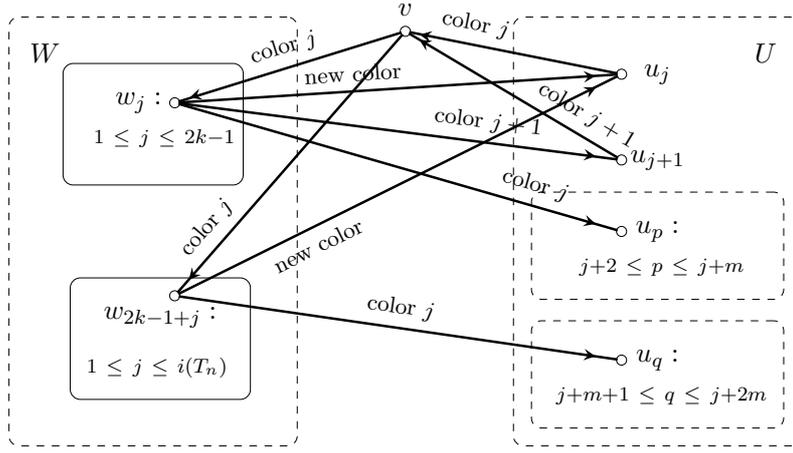}}
\vskip 1cm
		{\includegraphics[width=0.7\textwidth]{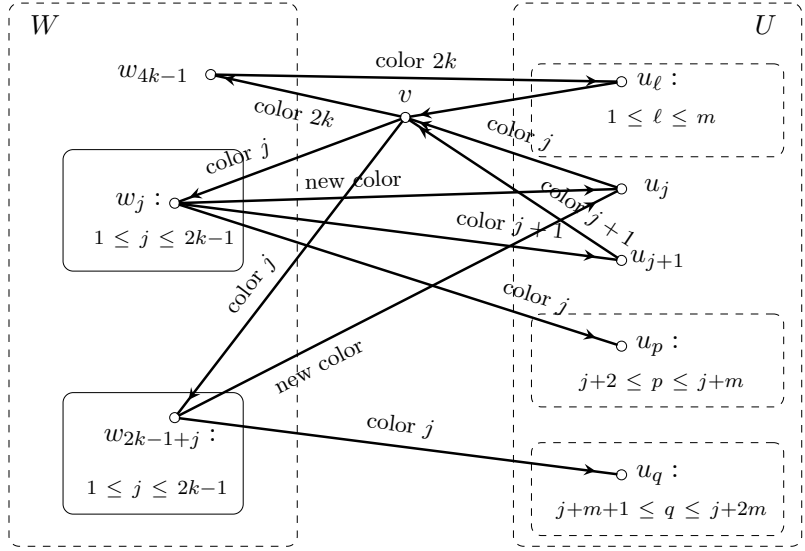}}
	\caption{Two digraphs described in case 2 of Example 3.\label{Figure3}}
\end{figure}

From this example we can see that the best bound of $2\Delta^{-mon}(T_{n})+\Delta^{+mon}(T_{n})$ is at most $m+1$. Namely, the gap between the bound in Theorem \ref{thm1+} in this case and the best possible bound is at most $1$.

\noindent\textbf{Example 4.} \quad This example shows that there is an arc-colored tournament $T_n$, in which $n$ is odd and $\Delta^{-mon}(T_{n})=\frac{n^{2}+n-3i(T_{n})^{2}}{12n}$, but there is no rainbow triangle in $T_n$.

Let $T_{n}$ be a regular tournament of order $11$ and $V(T_{n})=\{v_{0}, \ldots, v_{10}\}$. For each vertex $v_{i}\in V(T_{n})$, let $N^{+}(v_{i})=\{v_{i+1}, v_{i+3}, v_{i+4}, v_{i+5}, v_{i+9}\}$ (indices are taken modulo $11$). Let $C(v_{i}v_{i+1})=1$, $C(v_{i}v_{i+3})=2$, $C(v_{i}v_{i+4})=3$, $C(v_{i}v_{i+5})=4$ and $C(v_{i}v_{i+9})=5$ for $0\leq i\leq 10$. Then $\Delta^{-mon}(T_{n})=1=\frac{n+1}{12}$, and there is no rainbow triangle in $T_{n}$.

But this example is too special. So we give the next example to show that for regular tournaments the best possible upper bound of $\Delta^{-mon}(T_n)$ can not be larger than $\frac{n}{11}$.

\noindent\textbf{Example 5.} Replace each vertex $v_{i}$ of the tournament in Example 4 by a set $V_{i}$ of $k$ vertices. Let all arcs between $V_{i}$ and $V_{j}$ have the same directions and colors as the arc between $v_{i}$ and $v_{j}$. Add arcs to each $V_{i}$ to form a regular tournament $H_{i}$ of order $k$ and color all arcs in $H_i$ with a same new color. Denote the resulting graph by $D$. Then $D$ is a strongly connected regular tournament with $\Delta^{-mon}(D)=k=\frac{n}{11}$ and there is no rainbow triangle in $D$.

\section{Concluding remarks}
In Section \ref{sec:3}, we show the gap between the bound $\frac{(n-1-i(T_{n}))(n+1-i(T_{n}))}{4(n-1+i(T_{n}))}$ in Theorem \ref{thm1+} and the best possible bound is at most 1. We wonder whether this bound can be improved or there exists some examples showing the tightness of this bound.

For the existence of rainbow triangles in arc-colored regular tournaments, we conjecture that $\Delta^{-mon}(T_{n})<\frac{n}{11}$ is the best possible bound.

We also hope to make an improvement of the bound in Theorem \ref{thm2} for the existence of rainbow triangles in arc-colored tournaments with $i(T_{n})\neq 0$.


\begin{thebibliography}{10}
\bibitem{ADH: 2019}
{R. Aharoni, M. DeVos and R. Holzman},
\newblock Rainbow triangles and the Caccetta-H$\ddot{a}$ggkvist
conjecture,
\newblock {\em J. Graph Theory,} {\bf 92} (2019) 347--360.


\bibitem{Albert: 1995}
{M. Albert, A. Frieze and B. Reed},
\newblock Multicolored Hamilton cycles,
\newblock {\em Electron. J. Combin.,} {\bf 2} (1995),  \#R10.

\bibitem{Bai-Fujita-Zhang: 2018}
{Y. Bai, S. Fujita and S. Zhang},
\newblock Kernels by properly colored paths in arc-colored digraphs,
\newblock {\em Discrete Math.,} {\bf 341 (6)} (2018) 1523--1533.

\bibitem{Bai-Li-Zhang: 2018}
{Y. Bai, B. Li and S. Zhang},
\newblock Kernels by rainbow paths in arc-colored tournaments,
\newblock {\em Discrete Applied Mathematics,} (2019),
https://doi.org/10.1016/j.dam.2019.11.012.

\bibitem{Balogh: 2017}
{J. Bal$\acute{o}$gh, P. Hu, B. Lidick$\acute{y}$,
F. Pfender, J. Volec and M. Young}
\newblock Rainbow triangles in three-colored graphs,
\newblock {\em Journal of Combinatorial Theory,
Series B,} {\bf 126} (2017) 83--113.

\bibitem{Bang-Jensen: 2001}
{J. Bang-Jensen and G. Gutin},
\newblock Digraphs: Theory, Algorithms and Applications,
\newblock{\em Springer}, 2001.

\bibitem{Bland: 2016}
{A. Bland,}
\newblock Monochromatic sinks in $k$-arc colored tournaments,
\newblock {\em Graphs Combin.,} {\bf 32} (2016) 1279--1291.

\bibitem{Bondy: 2008}
{J.A. Bondy and U.S.R. Murty},
\newblock {\em Graph Theory},
\newblock {Springer}, 2008.

\bibitem{Ehard-Mohr: 2020}
{S. Ehard and E. Mohr},
\newblock Rainbow triangles and cliques in edge-colored graphs,
\newblock {\em European J. Combin., } {\bf 84} (2020) 103037.

\bibitem{Erdos: 1983}
{P. Erd\H{o}s, J. Ne\v{s}et\v{r}il and V. R\"{o}dl},
\newblock Some problems related to partitions of edges of a graph,
\newblock {\em Graphs and other combinatorial topics, Teubner, Leipzig }, (1983) 54--63.

\bibitem{Fox: 2015}
{J. Fox, A. Grinshpun and J. Pach},
\newblock The Erd\H{o}s-Hajnal conjecture for rainbow triangles,
\newblock {\em Journal of Combinatorial Theory, Series B,} {\bf 111} (2015) 75--125.

\bibitem{Frieze: 1993}
{A.M. Frieze and B.A. Reed},
\newblock Polychromatic Hamilton cycles,
\newblock {\em Discrete Math.,} {\bf 118} (1993) 69--74.


\bibitem{FLZ: 2017}
{S. Fujita, R. Li and S. Zhang},
\newblock Color degree and monochromatic degree conditions for short properly colored cycles in edge-colored graphs,
\newblock {\em J. Graph Theory, } {\bf 87} (2018) 362--373.

\bibitem{Fujita: 2014}
{S. Fujita, C. Magnant and K. Ozeki},
\newblock Rainbow generalizations of Ramsey theory: a survey,
\newblock{\em Graphs Combin.,} {\bf 26} (2010) 1--30.

\bibitem{FNXZ: 2019}
{S. Fujita, B. Ning, C. Xu and S. Zhang},
\newblock On sufficient conditions for rainbow cycles in edge-colored
graphs,
\newblock{\em Discrete Math.,} {\bf 342} (2019) 1956--1965.

\bibitem{Gallai: 1967}
{T. Gallai},
\newblock Transitiv orientierbare Graphen,
\newblock {\em Acta Math. Hungar.,} {\bf 18} (1967) 25--66.

\bibitem{Gyarfas: 2004}
{A. Gy\'{a}rf\'{a}s and G. Simonyi},
\newblock Edge colorings of complete graphs without tricolored triangles,
\newblock{\em J. Graph Theory,} {\bf 46} (2004) 211--216.

\bibitem{Hahn: 1986}
{G. Hahn and C. Thomassen},
\newblock Path and cycle sub-Ramsey numbers and an edge-colouring conjecture,
\newblock{\em Discrete Math.,} {\bf 62 (1)} (1986) 29--33.

\bibitem{HLO: 2017}
{C. Hoppen, H. Lefmann and K. Odermann},
\newblock On graphs with a large number of
edge-colorings avoiding a rainbow triangle,
\newblock{\em European J. Combin.,} {\bf 66} (2017) 168--190.



\bibitem{X.Li: 2008}
{M. Kano and X. Li},
\newblock Monochromatic and heterochromatic subgraphs in edge-colored graphs - a survey,
\newblock {\em Graphs Combin.,} {\bf 24} (2008) 237--263.

\bibitem{B.Li: 2014}
{B. Li, B. Ning, C. Xu and S. Zhang},
\newblock Rainbow triangles in edge-colored graphs,
\newblock {\em European J. Combin.,} {\bf 36} (2014) 453--459.

\bibitem{H.Li: 2013}
{H. Li},
\newblock Rainbow $C_3$'s and $C_4$'s in edge-colored graphs,
\newblock {\em Discrete Math.,} {\bf 313} (2013) 1893--1896.

\bibitem{Li-Wang: 2012}
{H. Li and G. Wang},
\newblock Color degree and heterochromatic cycles in edge-colored graphs,
\newblock {\em European J. Combin.,} {\bf 33} (2012) 1958--1964.

\bibitem{Li-Ning-Zhang: 2016}
{R. Li, B. Ning and S. Zhang},
\newblock Color degree sum conditions for rainbow triangles in edge-colored graphs,
\newblock {\em Graphs Combin.,} {\bf 32} (2016) 2001--2008.

\bibitem{Jin-Wang-Wang-Lv: 2020}
{Z. Jin, F. Wang, H. Wang and B. Lv},
\newblock Rainbow triangles in edge-colored Kneser graphs,
\newblock {\em Applied Mathematics and Computation,} {\bf 365} (2020) 124724.

\end{thebibliography}
\end{document}